%% file: Submission7_v12.tex
\documentclass[preprint,12pt]{elsarticle}
\usepackage{amsfonts,shuffle}
\usepackage{amssymb,tikz}
\input{macrosB}

\usepackage[latin1]{inputenc}
\usepackage{times}
\usepackage[T1]{fontenc}
\usepackage[all]{xy}
\usepackage{dsfont}\let\mathbb\mathds
\usepackage{MnSymbol}
\def\scal#1#2{\langle #1\bv#2 \rangle}
\def\ncp#1#2{#1\langle #2\rangle}

\newcommand{\mysetminusD}{\hbox{\tikz{\draw[line width=0.6pt,line cap=round] (3pt,0) -- (0,6pt);}}}

\begin{document}

\begin{frontmatter}



\title{Structure of Polyzetas and Explicit Representation\\on Transcendence Bases of Shuffle and Stuffle Algebras}


\author{V.C. Bui$^{\flat,\sharp}$, G. H. E. Duchamp$^{\sharp}$, V. Hoang Ngoc Minh$^{\diamondsuit,\sharp}$}

\address{$^{\flat}$Hue University of Sciences, 77 - Nguyen Hue street - Hue city, Vietnam\\
$^{\sharp}$Institut Galil\'ee, LIPN - UMR 7030, CNRS - Universit\'e Paris 13, F-93430 Villetaneuse, France,\\
$^{\diamondsuit
}$Universit\'e Lille II, 1, Place D\'eliot, 59024 Lille, France}

\begin{abstract}
Polyzetas, indexed by words, satisfy shuffle and quasi-shuffle identities. In this respect, one can explore the multiplicative and algorithmic (locally finite) properties of their generating series. In this paper, we construct pairs of bases in duality on which polyzetas are established in order to compute local coordinates in the infinite dimensional Lie groups where their non-commutative generating series live. We also propose new algorithms leading to the ideal of polynomial relations, homogeneous in weight, among polyzetas (the graded kernel) and their explicit representation (as data structures) in terms of irreducible elements.
\end{abstract}

\end{frontmatter}


\section{Introduction}
This paper will provide transparent arguments and proofs for results presented at the  International Symposium on Symbolic and Algebraic Computation conference, Bath, 6-9 July, 2015 \cite{ISSAC2015}.

For any composition of positive integers, $s=(s_1,\ldots,s_r)$, the polyzetas \cite{cartier2}
(also called multiple zeta values \cite{zagier}) are defined by the following convergent series
\begin{eqnarray}\label{polyzetas}
\zeta(s_1,\ldots,s_r):=\sum_{n_1>\ldots>n_r>0}n_1^{-s_1}\ldots n_r^{-s_r},&\mbox{for}&s_1>1.
\end{eqnarray}
The $\Q$-algebra generated by convergent polyzetas is denoted by $\cal Z$.

Any composition $s \in (\N_+)^r$ can be associated to words \cite{SLC43,MP} of the form
$x_0^{s_1-1}x_1\ldots x_0^{s_r-1}x_1$, defined on the alphabet $X=\{x_0,x_1\}$, or the form $y_{s_1} \ldots  y_{s_r}$, defined on the alphabet $Y=\{y_s\}_{s\ge1}$. The free monoids on these alphabets are respectively denoted by $X^*$ and $Y^*$. In this respect, the weight of the composition $s$, determined by $s_1+\ldots+s_r$, is also the weight of the word $y_{s_1} \ldots  y_{s_r}$ or the length of the word $x_0^{s_1-1}x_1\ldots x_0^{s_r-1}x_1$.

Using concatenation, shuffle and quasi-shuffle products, in Section \ref{Background}, 
\begin{enumerate}
\item We will recall the definition of Hopf algebras $(\mathbb{Q}\langle X\rangle,\bullet,1_{X^*},\Delta_{\shuffle}, {\tt e})$ and
$({\Q}\langle Y\rangle,\bullet,1_{Y^*},\Delta_{\ministuffle},{\tt e})$.

\item Equipping $X$ with the (total) ordering $x_0<x_1$ and denoting by $\Lyn X$, the set of Lyndon words over $X$,
the Poincar\'e-Birkhoff-Witt (PBW) basis $\{P_w\}_{w\in X^*}$ will be expanded over
the basis $\{P_l\}_{l\in \Lyn X}$, of the free Lie algebra $\mathcal{L}ie_\mathbb{Q}\langle X\rangle$.
Its dual basis $\{S_w\}_{w\in X^*}$ contains the pure transcendence basis of the algebra
$(\Q\langle X\rangle,\shuffle,1_{X^*})$ denoted by $\{S_l\}_{l\in\Lyn X}$ \cite{RT}.

\item Similarly, equipping $Y$ with the (total) ordering $y_1>y_2>y_3>\ldots$ and denoting by $\Lyn Y$
the set of Lyndon words over $Y$, the basis $\{\Pi_l\}_{l\in\Lyn Y}$, of the Lie algebra of primitive
elements\footnote{$P$ is a primitive element if $\Delta_{\stuffle}(P)=1_{Y^*}\otimes P+P\otimes 1_{Y^*}$.
This Lie algebra is isomorphic (but not equal) to the free Lie algebra.},
and its associated PBW-basis $\{\Pi_w\}_{w\in Y^*}$ will be proposed.
The dual basis $\{\Sigma_w\}_{w\in Y^*}$ is polynomial and contains also a pure transcendence basis of the algebra
$(\mathbb{Q}\langle Y\rangle,\stuffle,1_{Y^*})$ denoted by $\{\Sigma_l\}_{l\in\Lyn Y}$ \cite{BDM,acta,VJM}.

\item We then establish the two following expressions of the diagonal series
\begin{eqnarray}
{\cal D}_X:=\sum_{w\in X^*}w\otimes w=\prod_{l\in\Lyn X}^{\searrow}\exp(S_l\otimes P_l),\\
{\cal D}_Y:=\sum_{w\in Y^*}w\otimes w=\prod_{l\in\Lyn Y}^{\searrow}\exp(\Sigma_l\otimes\Pi_l).
\end{eqnarray}

\end{enumerate}

From these, in Section \ref{Structure},
\begin{enumerate}
\item We will consider two  generating series of polyzetas\footnote{In \eqref{Zstuff},
only {\it convergent} polyzetas arise then will not need any regularization.} \cite{acta,VJM,SLC43,MP}:
\begin{eqnarray}\label{Zstuff}
Z_{\minishuffle}:=\prod_{l\in\Lyn X\mysetminusD  X}^{\searrow}\exp(\zeta(S_l)P_l)&\mbox{and}&
Z_{\ministuffle}:=\prod_{l\in\Lyn Y\mysetminusD \{y_1\}}^{\searrow}\exp(\zeta(\Sigma_l)\Pi_l).
\end{eqnarray}
The coefficients of $Z_{\minishuffle}$ (resp. $Z_{\ministuffle}$) are obtained as the finite parts
of the asymptotic expansions of the polylogarithms $\{\mathrm{Li}_w\}_{w\in X^*}$
(resp. the harmonic sums $\{\mathrm{H}_w\}_{w\in Y^*}$), at $1$ (resp. at $+\infty$),
in the scale of comparison $\{(1-z)^{a}\log^b((1-z)^{-1})\}_{a\in\Z,b\in\N}$
(resp. $\{N^{a}\mathrm{H}_1^b(N)\}_{a\in\Z,b\in\N}$, where $\mathrm{H}_1(N)$ is the classic harmonic sum
$1+1/2+\ldots+1/N$) \cite{acta}.

\item We have also defined a third one, $Z_{\gamma}$ \cite{jsc}, which satisfies, via the $\stuffle$-extended
Sch\"utzenberger's factorization on the completed quasi-shuffle Hopf algebra \cite{acta,VJM}~,
\begin{eqnarray}\label{Z_gamma}
Z_{\gamma}=e^{\gamma y_1}Z_{\ministuffle}.
\end{eqnarray}
The coefficients of $Z_{\gamma}$ are obtained as the finite parts of the asymptotic expansions of $\{\mathrm{H}_w\}_{w\in Y^*}$,
in the scale of comparison $\{N^{a}\log^b(N)\}_{a\in\Z,b\in\N}$.

In \eqref{Z_gamma}, $\gamma$ denotes the Euler's constant \cite{jsc}.

\item In order to identify the local coordinates of $Z_{\shuffle}$ (and $Z_{\ministuffle}$),
on a group of associators \cite{acta,VJM}, we will rely on the following comparison (see \cite{jsc})
\begin{eqnarray}\label{regularization}
Z_{\gamma}=B(y_1)\pi_Y(Z_{\shuffle}),\mbox{ where }
B(y_1)=\exp\biggl(\gamma y_1-\sum_{k\geq2}{\frac{(-1)^{k-1}\zeta (k)}{k}y_1^k}\biggr).
\end{eqnarray}
Here, $\pi_Y$ is a linear projection from $\mathbb{Q}\oplus\mathbb{Q}\llangle X\rrangle x_1$
to $\mathbb{Q}\llangle Y\rrangle$,
mapping $x_0^{s_1-1}x_1\ldots x_0^{s_r-1}x_1$ to $y_{s_1}\ldots y_{s_r}$, and $\pi_X$ denotes its inverse.

By {\it cancellation} \cite{acta,VJM}, \eqref{Z_gamma} and \eqref{regularization} yield  the following identity
\begin{eqnarray}\label{regularization2}
Z_{\ministuffle}=B'(y_1)\pi_Y(Z_{\shuffle}),\mbox{ where }B'(y_1)=\exp\biggl(\sum_{k\geq2}{\frac{(-1)^{k-1}\zeta (k)}{k}y_1^k}\biggr).
\end{eqnarray}

\item Simultaneously, algorithms will be also implemented in Maple to represent
polyzetas\footnote{The Maple program runs on a computer Core(TM)i5-4210U CPU @ 1.70GHz
and obtains results up to weight $12$ \cite{program}.} in terms of irreducible polyzetas
producing {\it algebraic} relations among  the local coordinates $\{\zeta(S_l)\}_{l\in\Lyn X\mysetminusD  X}$
(and $\{\zeta(\Sigma_l)\}_{l\in\Lyn Y\mysetminusD \{y_1\}}$) \cite{program}.
\end{enumerate}

To end this section, let us point out some crucial points of our purpose~:

\begin{enumerate}
\item Similar tables\footnote{They form a Gr\"obner basis of the ideal of polynomial relations among
the convergent polyzetas and the ranking of this basis is based mainly on the order of Lyndon words
\cite{bigotte,MP,elwardi}. For that, this basis is also called Gr\"obner-Lyndon basis.}
for $\{\zeta(l)\}_{l\in\Lyn X\tiny \mysetminusD X}$ have been obtained up to weight
$10$ \cite{MP}, $12$ \cite{bigotte} and $16$ \cite{elwardi}.
These differ from the zig-zag relation among the moulds of  formal polyzetas, due to Ecalle \cite{ecalle},
{\it i.e.} the commutative generating series of {\it symbolic} polyzetas (Boutet de Monvel \cite{boutet}
and Racinet \cite{racinet} have also given equivalent relations for the noncommutative generating series
of symbolic polyzetas, see also \cite{cartier2}) producing {\it linear} relations and which base themselves on
{\it regularized double shuffle relation} \cite{blumlein,hoffman,kaneko}
and different from identities among associators, due to Drinfel'd \cite{drinfeld1,drinfeld2,Furusho}.

\item In the classical theory of finite-dimensional Lie groups,
any ordered basis of Lie algebra provides a system of local coordinates
in suitable neighborhood of the group unity via an ordered product of one-parameter groups
corresponding to the ordered basis \cite{weinorman2}.
In this work, we get a perfect analogue of this picture for Hausdorff groups,
through {\it Sch\"utzenberger's factorization}, this doesn't depend on regularization (see the next remark) \cite{acta,VJM}.

Moreover, through the bridge equation \eqref{regularization} relating two elements on these groups
and by identification of local coordinates, in infinite dimension, of  their L.H.S. and R.H.S.
(which involve only convergent polyzetas) we get again a confirmation of Zagier's conjecture, up to weight $12$.
This is not a consequence of regularized double-shuffle relation (see the next remarks).

\item Of course, the generating series given in \eqref{Zstuff} and \eqref{Z_gamma} induce,
as already shown in \cite{acta,VJM}, three morphisms of (shuffle and quasi-shuffle) algebras,
studied earlier in \cite{SLC43,SLC44,jsc} and constructed in \cite{acta,VJM}
\begin{eqnarray}
\zeta_{\minishuffle}:&(\QX,\shuffle,1_{X^*})&\longrightarrow ({\cal Z},\times,1),\\
\zeta_{\ministuffle}:&(\mathbb{Q}\langle Y\rangle,\ministuffle,1_{Y^*})&\longrightarrow ({\cal Z},\times,1),\\
\gamma_{\bullet}:&(\mathbb{Q}\langle Y\rangle,\ministuffle,1_{Y^*})&\longrightarrow ({\cal Z},\times,1),
\end{eqnarray}
which satisfy, for any $u=x_0^{s_1-1}x_1\ldots x_0^{s_r-1}x_1\in x_0X^*x_1$ and $v=\pi_Y(u)$,
\begin{eqnarray}
\zeta_{\minishuffle}(u)=\zeta_{\ministuffle}(v)=\gamma_{v}=\zeta(s_1,\ldots,s_r)
\end{eqnarray}
and the generators of length (resp. weight) one, for $X^*$ (resp. $Y^*$), satisfy (see \eqref{Zstuff} and \eqref{Z_gamma})
\begin{eqnarray}\label{zeta1}
\zeta_{\minishuffle}(x_0)=\zeta_{\shuffle}(x_1)=\zeta_{\ministuffle}(y_1)=0&\mbox{and}&\gamma_{y_1}=\gamma.
\end{eqnarray}

Hence, $\zeta_{\shuffle},\zeta_{\stuffle}$ and $\gamma_{\bullet}$ are characters
of (shuffle and quasi-shuffle) Hopf algebras, and their graphs,
written as series, respectively read \cite{jsc,SLC44}
\begin{eqnarray}
\sum_{w\in X^*}\zeta_{\minishuffle}(w)w=Z_{\shuffle},&
\Sum_{w\in Y^*}\zeta_{\ministuffle}(w)w=Z_{\ministuffle},&
\sum_{w\in Y^*}\gamma_w\;w=Z_{\gamma}
\end{eqnarray}
and\footnote{They are group-like~: $\Delta_{\shuffle}(Z_{\shuffle})=Z_{\minishuffle}\otimes Z_{\minishuffle}$,
$\Delta_{\ministuffle}(Z_{\ministuffle})=Z_{\ministuffle}\otimes Z_{\ministuffle}$,
$\Delta_{\gamma}(Z_{\gamma})=Z_{\gamma}\otimes Z_{\gamma}$.}
$Z_{\minishuffle}=(\zeta_{\minishuffle}\otimes\mathrm{Id}_{X^*}){\cal D}_X,
Z_{\ministuffle}=(\zeta_{\ministuffle}\otimes\mathrm{Id}_{Y^*}){\cal D}_Y,
Z_{\gamma}=(\gamma_{\bullet}\otimes\mathrm{Id}_{Y^*}){\cal D}_Y.$

\item By \eqref{Zstuff}, for any $u,v\in\Lyn X\mysetminusD  X$ and $u'=\pi_Y(u),v'=\pi_Y(y)$, one has
\begin{eqnarray}
\zeta_{\minishuffle}(u)\zeta_{\minishuffle}(v)=\zeta_{\minishuffle}(u\shuffle v)
&\mbox{and}&
\zeta_{\ministuffle}(u')\zeta_{\ministuffle}(v')=\zeta_{\ministuffle}(u'\stuffle v').
\end{eqnarray}
By \eqref{regularization2}, for any $l\in\Lyn X\mysetminusD  X$ and $l'=\pi_Y(l)$, one has, on the other hand
\begin{enumerate}
\item[i)] $\zeta_{\minishuffle}(x_1\shuffle l-x_1l)=-\zeta_{\minishuffle}(x_1l)=-\scal{Z_{\minishuffle}}{x_1l}$,
\item[ii)] $\zeta_{\ministuffle}(y_1\stuffle l'-y_1l')=-\zeta_{\ministuffle}(y_1l')=-\scal{Z_{\ministuffle}}{y_1l'}$,
\item[iii)] $\scal{B'(y_1)}{y_1}=0$.
\end{enumerate}
This means that since \eqref{regularization2} is equivalent to \eqref{regularization}, for the quasi-shuffle product,
the regularization to $\gamma$ is equivalent to the regularization to $0$ \cite{acta,VJM}
and this yields immediately the family of regularized double shuffle relations considered in
\cite{bigotte,blumlein,racinet,kaneko,MP,elwardi} (see also \cite{cartier2,boutet,hoffman,IKZ,waldschmidt}).

Our method is then different from \cite{cartier2,boutet,waldschmidt} in which their authors suggest
the {\it simultaneous} regularization of the divergent polyzeta $\zeta(1)$ to the indeterminate $T$,
{\it i.e.} $\zeta_{\shuffle}(x_0)=\zeta_{\shuffle}(x_1)=\zeta_{\ministuffle}(y_1)=T$
(to compare with \eqref{zeta1}). Since $T$ is transcendent over $\Q$ then
it can be suitable to be specialized to $0$, as effectively done in \cite{racinet,kaneko}
and, by this way, relations among polyzetas are {\it formally} obtained depending mainly
on numerical values\footnote{Since the $\Q$-algebra of polyzetas is not a $\Q[T]$-algebra,
how then can we determine these values ?} of $T$.
\end{enumerate}

\section{Background}\label{Background}
\subsection{Generalities}
Let $Y=\{y_s\}_{s\ge1}$ be an infinite alphabet with the total  order $y_1>y_2>\ldots$.\\ $Y^*$ denotes the free monoid on $Y$ which admits the empty word, denoted by $1_{Y^*}$, as neutral element.

Let us define the commutative  product on\footnote{$\Q Y$ denotes the $\Q$-vector space generated by the alphabet $Y$, as a basis.} $\Q Y$, denoted by $\mu$ (see \cite{BDHMT,orsay}),
\begin{eqnarray}
\forall y_s,y_t\in Y,&&\mu(y_s, y_t)=y_{s+t},
\end{eqnarray}
or its dual coproduct, $\Delta_{\mu}$, defined by
\begin{eqnarray}
\forall y_s\in Y,&&\Delta_{\mu}y_s=\sum_{i=1}^{s-1}y_i\otimes y_{s-i}
\end{eqnarray}
satisfying,
\begin{eqnarray}
\forall x,y,z\in Y,&&\scal{\Delta_\mu x}{y\otimes z}=\scal{x}{\mu(y,z)}.
\end{eqnarray}

Let $\ncp{\Q}{Y}$ denote the space of polynomials on the alphabet $Y$ equipped by
\begin{enumerate}
\item The concatenation $\bullet$ (or by its associated coproduct, $\Delta_{\bullet}$).
\item The {\it shuffle} product, {\it i.e.} the commutative product defined by \cite{RT},
for any $y_s,y_t\in Y$ and $u,v,w\in Y^*$
\begin{eqnarray}
w\shuffle 1_{Y^*}&=&1_{Y^*}\shuffle w=w,\cr
y_su\shuffle y_tv&=&y_s(u\shuffle y_tv)+y_t(y_su\shuffle v)
\end{eqnarray}
or by its associated coproduct, $\Delta_{\minishuffle}$,  defined, on the letters by, 
\begin{eqnarray}
\forall y_s\in Y,&&\Delta_{\minishuffle}y_s=y_s\otimes1_{Y^*}+1_{Y^*}\otimes y_s
\end{eqnarray}
and extended so as to make it a homomorphism for the concatenation product. It satisfies
\begin{eqnarray}
\forall u,v,w\in Y^*,&&\scal{\Delta_{\minishuffle}w}{ u\otimes v}=\scal{w}{ u\shuffle v}.
\end{eqnarray}

\item The {\it quasi-shuffle} product, {\it i.e.} the commutative product defined by \cite{hoffman},
for any $y_s,y_t\in Y$ and $u,v,w\in Y^*$, 
\begin{eqnarray}
w\ministuffle 1_{Y^*}&=&1_{Y^*}\ministuffle w=w,\cr
y_su\ministuffle y_tv&=&y_s(u\ministuffle y_tv) + y_t(y_su\ministuffle v)+ \mu(y_s,y_t)(u\ministuffle v)
\end{eqnarray}
or by its associated coproduct, $\Delta_{\ministuffle}$,  defined, on the letters by,
\begin{eqnarray}
\forall y_s\in Y,&&\Delta_{\ministuffle}y_s=\Delta_{\minishuffle}y_s+\Delta_\mu y_s
\end{eqnarray}
and extended so as to make it a homomorphism for the concatenation product. It satisfies
\begin{eqnarray}
\forall u,v,w\in Y^*,&&\langle\Delta_{\ministuffle}w\mid u\otimes v\rangle=\langle w\mid u\ministuffle v\rangle.
\end{eqnarray}
Note that $\Delta_{\minishuffle}$ and $\Delta_{\ministuffle}$ are morphisms from $\ncp{\Q}{Y}$ for the concatenation 
but $\Delta_{\mu}$ is not (for example $\Delta_{\mu}(y_1^2)=y_1\otimes y_1$, whereas $\Delta_{\mu}(y_1)^2=0$).
\end{enumerate}
Hence, with the counit $\tt e$ defined by ${\tt e}(P)=\scal{P}{1_{Y^*}}$ (for any $P\in{\Q}\langle Y\rangle$).
We get two pairs of mutually dual bialgebras
\begin{eqnarray}
\calH_{\minishuffle}=({\Q}\langle Y\rangle,\bullet,1_{Y^*},\Delta_{\shuffle},{\tt e}),&&
\calH_{\minishuffle}^{\vee}=({\Q}\langle Y\rangle,\shuffle,1_{Y^*},\Delta_{\bullet},{\tt e}),\\
\calH_{\ministuffle}=({\Q}\langle Y\rangle,\bullet,1_{Y^*},\Delta_{\ministuffle},{\tt e}),&&
\calH_{\ministuffle}^{\vee}=({\Q}\langle Y\rangle,\ministuffle,1_{Y^*},\Delta_{\bullet},{\tt e}).
\end{eqnarray}
Let us then consider the following diagonal series\footnote{Of course, we have (set theoretically) $\mathcal{D}_{\minishuffle}=\mathcal{D}_{\ministuffle}$,
but their structural treatments will be different.}
\begin{eqnarray}
\mathcal{D}_{\minishuffle}=\sum_{w\in Y^*}w\otimes w&\mbox{and}&\mathcal{D}_{\ministuffle}=\sum_{w\in Y^*}w\otimes w.
\end{eqnarray}
Here, for the algebras where live in $\mathcal{D}_{\minishuffle}$ and $\mathcal{D}_{\ministuffle}$,
the operation on the right factor of the tensor product is the concatenation,
and the operation on the left factor is the shuffle and the quasi-shuffle, respectively.

By the Cartier-Quillen-Milnor and Moore (CQMM) theorem \cite{BDHMT,CMM}, the connected $\N$-graded,
co-commutative Hopf algebra $\mathcal{H}_{\minishuffle}$ is isomorphic to the enveloping algebra of the Lie algebra of its  primitive elements which is ${\calL ie}_{\Q}\langle Y\rangle$~:
\begin{eqnarray}
\calH_{\minishuffle}\cong\mathcal{U}({\calL ie}_{\Q}\langle Y\rangle)
&\mbox{and}&
\calH_{\minishuffle}^{\vee}\cong\mathcal{U}({\calL ie}_{\Q}\langle Y\rangle)^{\vee}.
\end{eqnarray}
Hence, denoting by $(l_1,l_2)$ the standard factorization\footnote{A pair of Lyndon words $(l_1,l_2)$
is called the standard factorization of $l$ if $l=l_1l_2$ and $l_2$ is the smallest nontrivial proper right factor of $l$
(for the lexicographic order) or, equivalently its longest such.} of $l\in\Lyn Y\mysetminusD  Y$, let us consider
\begin{enumerate}
\item The PBW basis $\{P_w\}_{w\in Y^*}$ constructed recursively as follows \cite{RT}
\begin{equation}\label{basisP}
\left\{\begin{array}{llll}
P_{y_s}&=&y_s&\mbox{for }y_s\in Y,\\
P_{l}&=&[P_{l_1},P_{l_2}]&\mbox{for }l\in\Lyn Y\mysetminusD  Y,\ st(l)=(l_1,l_2),\\
P_{w}&=&P_{l_1}^{i_1}\ldots P_{l_k}^{i_k}
&{\mbox{for }w=l_1^{i_1}\ldots l_k^{i_k},\mbox{ with}\atop
l_1,\ldots,l_k\in\Lyn Y,\ l_1>\ldots>l_k.}
\end{array}\right.
\end{equation}
\begin{example}
$i)$ Considering on the alphabet $Y:$
\begin{eqnarray*}
P_{y_1}&=&y_1,\quad P_{y_2}= y_2,\\
P_{y_2y_1}&=&y_2y_1-y_1y_2,\\
P_{y_3y_1y_2}&=&y_3y_1y_2-y_2y_3y_1+y_2y_1y_3-y_1y_3y_2.
\end{eqnarray*}
$ii)$ Considering on the alphabet $X=\{x_0,x_1\},\ x_0<x_1:$
\begin{eqnarray*}
P_{x_1}&=&x_1,\quad P_{x_0x_1}= x_0x_1-x_1x_0,\\
P_{x_0x_1^2}&=&x_0y_1^2-2x_1x_0x_1+y_1^2x_0,\\
P_{x_0^2x_1^2x_0x_1}&=&x_0^2x_1^2x_0x_1-x_0^2x_1^3x_0+2x_0x_1x_0x_1^2x_0+2x_1x_0x_1x_0x_0x_1\\
&-&x_1^2x_0^3x_1+x_1^2x_0^2x_1x_0-x_0x_1x_0^2x_1^2-2x_0x_1^2x_0x_1x_0+x_0x_1^3x_0^2\\
&+&x_1x_0^3x_1^2-2x_1x_0^2x_1x_0x_1-x_1x_0x_1^2x_0^2.
\end{eqnarray*}
\end{example}

\item and, by duality\footnote{The dual family, {\it i.e.} the set of coordinates forming a basis 
in the algebraic dual which is here the space of noncommutative series,
but as the enveloping algebra under consideration is graded in finite dimensions (by the multidegree),
these series are in fact multi-homogeneous polynomials.}, the basis
$\{S_w\}_{w\in Y^*}$ of $({\Q}\langle Y\rangle,\shuffle)$, {\it i.e.}
\begin{eqnarray}
\forall u,v\in Y^*,\langle P_u\mid S_v\rangle=\delta_{u,v}.
\end{eqnarray}
This linear basis can be computed recursively as follows \cite{RT}
\begin{equation}\label{basisS}
\left\{\begin{array}{llll}
S_{y_s}&=&y_s,&\mbox{for }y_s\in Y,\\
S_l&=&y_sS_u,&\mbox{for }l=y_su\in\Lyn Y,\\
S_w&=&\Frac{S_{l_1}^{\shuffle i_1}\shuffle\ldots\shuffle S_{l_k}^{\shuffle i_k}}{i_1!\ldots i_k!}
&{\mbox{for }\displaystyle w=l_1^{i_1}\ldots l_k^{i_k},\mbox{ with}\hfill\atop\displaystyle l_1,\ldots,l_k\in\Lyn Y,\ l_1>\ldots>l_k.}
\end{array}\right.
\end{equation}

\begin{example}\label{examS_l}
$i)$ Considering on the alphabet $Y:$
\begin{eqnarray*}
S_{y_1}&=&y_1,\\
S_{y_2}&=&y_2,\\
S_{y_2y_1}&=&y_2y_1,\\
S_{y_3y_1y_2}&=&y_3y_2y_1+y_3y_1y_2.
\end{eqnarray*}
$ii)$ Considering on the alphabet $X:$
\begin{eqnarray*}
S_{x_1}&=&x_1,\\
S_{x_0x_1}&=&x_0x_1,\\
S_{x_0x_1^2}&=&x_0x_1^2,\\
S_{x_0^2x_1^2x_0x_1}&=&x_0^2x_1^2x_0x_1+3x_0^2x_1x_0x_1^2+6x_0^3x_1^3.
\end{eqnarray*}
\end{example}
\end{enumerate}

Similarly, by CQMM theorem, the connected $\mathbb{N}$-graded, co-commutative Hopf algebra $\mathcal{H}_{\ministuffle}$
is isomorphic to the enveloping algebra of  its  primitive elements:
\begin{equation}
\mathrm{Prim}(\calH_{\ministuffle})=\mathrm{Im}(\pi_1)=\mathrm{span}_{\Q}\{\pi_1(w)\vert{w\in Y^*}\},\\
\end{equation}
where, for any $w\in Y^*,\pi_1(w)$ is obtained as follows  \cite{acta,VJM}
\begin{eqnarray}\label{pi1}
\pi_1(w) =w+\sum_{k=2}^{(w)}\frac{(-1)^{k-1}}k\sum_{u_1,\ldots,u_k\in Y^+}\langle w\mid u_1\ministuffle\ldots\ministuffle u_k\rangle\;u_1\ldots u_k.
\end{eqnarray}
Note that \eqref{pi1} is equivalent to the following identity
\begin{equation}\label{exp1}
w=\sum_{k\ge0}\frac1{k!}\sum_{u_1,\ldots,u_k\in Y^*}
\langle w\mid u_1\ministuffle\ldots\ministuffle u_k\rangle\;\pi_1(u_1)\ldots\pi_1(u_k).
\end{equation}
In particular, for any $y_s\in Y$, the primitive polynomial $\pi_1(y_s)$ is given by
\begin{eqnarray}\label{pi1bis}
\pi_1(y_s)=y_s+\sum_{i=2}^s\frac{(-1)^{i-1}}{l}\sum_{j_1,\ldots,j_i\ge1,j_1+\ldots+j_i=s}y_{j_1}\ldots y_{j_i}.
\end{eqnarray}
\begin{example}
$\pi_1(y_1)=y_1,\ \pi_1(y_2)=y_2-\frac12y_1^2,\ \pi_1(y_3)=y_3-\frac12(y_1y_2+y_2y_1)+\frac13y_1^3$.
\end{example}
As previously, the expressions (\ref{pi1bis}) are equivalent to
\begin{eqnarray}
y_s=\sum_{i\ge1}\frac{1}{i!}\sum_{s_1+\ldots +s_i=s}\pi_1(y_{s_1})\ldots\pi_1(y_{s_i}),\quad y_s\in Y\label{pi1ter}.
\end{eqnarray}
\begin{example}
$$\begin{array}{rcl}
y_1&=&\pi_1(y_1),\\
y_2&=&\pi_1(y_2)+\frac{1}{2!}\pi_1(y_1)^2,\\
y_3&=&\pi_1(y_3)+\frac{1}{2!}\left(\pi_1(y_1)\pi_1(y_2)+\pi_1(y_2)\pi_1(y_1)\right)+\frac{1}{3!}\pi_1(y_1)^3.
\end{array}$$
\end{example}

Now let us consider the (endo-)morphism of algebras
$\phi:(\Q\pol{Y},\bullet,1)\rightarrow (\Q\langle Y\rangle,\bullet, 1)$ satisfying $\phi(y_k)=\pi_1(y_k)$; it can be shown that $\phi$ is an automorphism of 
$\ncp{\Q}{Y}$. Then we have \cite{VJM},
\begin{enumerate}
\item[i)] $\phi$ realizes an isomorphism from the bialgebra $(\Q\pol{Y},\bullet,\Delta_{\minishuffle},{\tt e})$
to the bialgebra $({\Q}\langle Y\rangle,\bullet,\Delta_{\ministuffle},{\tt e})$.
\item[ii)] In particular, we have the following commutative diagram
$$\def\commutatif{\ar@{}[rd]|{\circlearrowleft}}
      \xymatrix{\relax
    \Q\pol{Y}\ar[r]^-{\Delta_{\minishuffle}} \ar[d]_{\phi}& \Q\pol{Y}\otimes\Q\pol{Y}\ar[d]^-{\phi\otimes\phi}\\
    {\Q}\langle Y\rangle\ar[r]_-{\Delta_{\ministuffle}}&{\Q}\langle Y\rangle\otimes{\Q}\langle Y\rangle.}$$
\item[iii)]  $\calH_{\ministuffle}\cong\mathcal{U}(\mathrm{Prim}(\calH_{\ministuffle}))$ and
$\calH_{\ministuffle}^{\vee}\cong\mathcal{U}(\mathrm{Prim}(\calH_{\ministuffle}))^{\vee}$.
\item[iv)] The dual bases $\{\Pi_w\}_{w\in Y^*}$ and $\{\Sigma_w\}_{w\in Y^*}$ of respectively 
$\mathcal{U}(\mathrm{Prim}(\calH_{\ministuffle}))$ and $\mathcal{U}(\mathrm{Prim}(\calH_{\ministuffle}))^{\vee}$
can be obtained as images, respectively by $\phi$ and $\check\phi^{-1}$, of respectively $\{P_w\}_{w\in Y^*}$ and $\{S_w\}_{w\in Y^*}$.
\end{enumerate}
More precisely,
\begin{enumerate}
\item The PBW basis $\{\Pi_w\}_{w\in Y^*}$ for $\mathcal{U}(\mathrm{Prim}(\calH_{\ministuffle}))$ can be cons\-truc\-ted recursively as follows \cite{,BDM,acta,VJM}
\begin{equation}
\left\{\begin{array}{llll}
\Pi_{y_s}&=&\pi_1(y_s)& \mbox{for }y_s\in Y,\\
\Pi_{l}&=&[\Pi_{l_1},\Pi_{l_2}]&\mbox{for }l\in\Lyn Y\mysetminusD  Y,\ st(l)=(l_1,l_2),\\
\Pi_{w}&=&\Pi_{l_1}^{i_1}\ldots\Pi_{l_k}^{i_k}
&{\displaystyle\mbox{for }w=l_1^{i_1}\ldots l_k^{i_k},\mbox{ with}\hfill\atop\displaystyle l_1,\ldots,l_k\in\Lyn Y,\ l_1>\ldots>l_k.}
\end{array}\right.
\end{equation}

\begin{example}
$$\begin{array}{rcl}
\Pi_{y_1}&=&y_1,\\
\Pi_{y_2}&=&y_2-\frac{1}{2}y_1^2,\\
\Pi_{y_2y_1}&=&y_2y_1-y_1y_2,\\
\Pi_{y_3y_1y_2}&=&y_3y_1y_2-\frac{1}{2}y_3y_1^3-y_2y_1^2y_2
+\frac{1}{4}y_2y_1^4-y_1y_3y_2+\frac{1}{2}y_1y_3y_1^2
+\frac{1}{2}y_1^2y_2^2\\
&-&\frac{1}{2}y_1^2y_2y_1^2-y_2y_3y_1+\frac{1}{2}y_2^2y_1^2+y_2y_1y_3+\frac{1}{2}y_1^2y_3y_1-\frac{1}{2}y_1^3y_3+\frac{1}{4}y_1^4y_2.
\end{array}$$
\end{example}
\item and, by duality, the  basis $\{\Sigma_w\}_{w\in Y^*}$ of (${\Q}\langle Y\rangle$,\stuffle), {\it i.e.}
\begin{eqnarray}
\forall u,v\in Y^*,&&\langle\Pi_u\mid \Sigma_v\rangle=\delta_{u,v}.
\end{eqnarray}
This linear basis can be computed recursively as follows  \cite{BDM,acta,VJM}
\begin{equation}\label{basisSigma}
\left\{\begin{array}{ll}
\Sigma_{y_s}=y_s,
&\mbox{for }y_s\in Y,\\
\Sigma_l\;=
\Sum_{(\smallstar)}\frac{1}{i!}y_{s_{k_1}+\ldots+s_{k_i}}\Sigma_{l_1\ldots l_n},
&\mbox{for }l=y_{s_1}\ldots y_{s_k}\in\Lyn Y,\\
\Sigma_w=\displaystyle\frac{\Sigma_{l_1}^{\ministuffle i_1}\ministuffle\ldots\ministuffle\Sigma_{l_k}^{\ministuffle i_k}}{i_1!\ldots i_k!},
&{\displaystyle\mbox{for }w=l_1^{i_1}\ldots l_k^{i_k},\mbox{ with}\hfill\atop\displaystyle l_1,\ldots,l_k\in\Lyn Y,\ l_1>\ldots>l_k.}
\end{array}\right.
\end{equation}
In $(\smallstar)$, the sum is taken over all $\{k_1,\ldots,k_i\}\allowbreak \subset \allowbreak\{1,\ldots,k\}$
and all $l_1\ge\ldots\ge l_n$ such that $(y_{s_1},\ldots,\allowbreak y_{s_k})\stackrel{*}{\Leftarrow}(y_{s_{k_1}},\ldots,y_{s_{k_i}},l_1,\ldots,l_n)$,
where $\stackrel{*}{\Leftarrow}$ denotes the transitive closure of the relation on standard sequences, denoted by $\Leftarrow$ \cite{BDM}.

Using Example \ref{examS_l}.ii), we have in general, for any $l\in\Lyn Y,\pi_X(\Sigma_l)\neq S_{\pi_Xl}$
(resp. $\Lyn X\setminus\{x_0\},\pi_Y(S_l)\neq\Sigma_{\pi_Yl}$) \cite{acta,VJM}:

\begin{example}\label{different}
$$\begin{array}{|c|c||c|c|}
\hline
l\in \Lyn Y&\Sigma_l&\pi_X(l)\in \Lyn X&\pi_YS_{\pi_X(l)}\\
\hline
y_1&y_1&x_1&y_1\\
y_2&y_2&x_0x_1&y_2\\
y_2y_1&y_2y_1+\frac{1}{2}y_3&x_0x_1^2&y_2y_1\\
{y_3y_1y_2}&y_3y_2y_1+y_3y_1y_2+y_3^2&x_0^2x_1^2x_0x_1&y_3y_1y_2+3y_3y_2y_1\\
&+\frac{1}{2}y_4y_2+\frac{1}{2}y_5y_1+\frac{1}{3}y_6&&+6y_4y_1^2\\
\hline
\end{array}$$
\end{example}
\end{enumerate}

\subsection{Local coordinates}
Following Wei-Norman's theorem \cite{weinorman2}, we know that, for a given (finite dimensional) ${\bf k}$-Lie group\footnote{Real (with ${\bf k}=\mathbb{R}$) or complex (with ${\bf k}=\mathbb{C}$).} $G$, its Lie algebra $\mathfrak{g}$, and a basis $B=(b_i)_{1\leq i\leq n}$ of $\mathfrak{g}$, there exists a neighbourhood $W$ of $1_G$ (in $G$) and $n$ {\it local coordinate} ${\bf k}$-valued analytic functions 
\begin{eqnarray*}
W\rightarrow{\bf k},&&(t_i)_{1\le i\le n}
\end{eqnarray*}
such that, for all $g\in W$,
\begin{eqnarray*}
g=\prod_{1\leq i\le n}^{\rightarrow}e^{t_i(g)b_i}=e^{t_1(g)b_1}\dots e^{t_n(g)b_n}.
\end{eqnarray*}
The proof relies on the fact that, $(t_1,\ldots,t_n)\rightarrow e^{t_1(g)b_1}\dots e^{t_n(g)b_n}$ is a local diffeomorphism from ${\bf k}^n$ to $G$ at a neighbourhood of $0$.
\begin{example}[Wei-Norman in finite dimensions]
Let $M\in Gl_+(2,\R)$ ($Gl_+(2,\R)$ denote the connected component of $1$ in the Lie group\footnote{It is the group of matrices with positive determinant.} $Gl(2,\R)$
\begin{equation*}
M=
\begin{pmatrix}
 a_{11} & a_{12}\\
 a_{21} & a_{22} 
\end{pmatrix}
\end{equation*}
In order to perform the decomposition, we will ``go back to identity'' by computing $MTDU=I$, where $I$ stands for the identity matrix,
$T$ is upper unitriangular, $D$ diagonal strictly positive and $U$ unitary,
then $M=U^{-1}D^{-1}T^{-1}$ will be the Iwasawa \cite{Bourbaki} decomposition of $M$. The decomposition algorithm goes in three
steps as follows (step 4 is a summary)
\begin{enumerate}
\item ({\bf Orthogonalization})
We perform block-computation on the columns of $M$ to obtain an orthogonal matrix
\begin{equation*}
M
\longrightarrow
\begin{pmatrix}
 a_{11} & a_{12}\\
 a_{21} & a_{22} 
\end{pmatrix}
\begin{pmatrix}
 1 & t_1\\
 0 & 1 
\end{pmatrix}
=MT
=
\begin{pmatrix}
 a_{11}^{(1)} & a_{12}^{(1)}\\
 a_{21}^{(1)} & a_{22}^{(1)} 
\end{pmatrix}
=\begin{pmatrix}
C_1^{(1)} &
C_2^{(1)} 
\end{pmatrix}
=M_1.
\end{equation*}
the both of columns are orthogonal if
$t_1=-\frac{a_{11}a_{12}+a_{21}a_{22}}{a_{11}^2+ a_{21}^2}.
$
\item ({\bf Normalization}) We normalize $M_1$,
\begin{eqnarray*}
M_2&=&
\begin{pmatrix}
C_1^{(1)} &
C_2^{(1)} 
\end{pmatrix}
\begin{pmatrix}
\frac{1}{||C_1^{(1)}||} & 0\cr
0 & \frac{1}{||C_2^{(1)}||} 
\end{pmatrix}
=M_1D\cr
&=&M_1e^{-\log(||C_1^{(1)}||)
\begin{pmatrix}
1 & 0\cr
0 & 0 
\end{pmatrix}
-\log(||C_2^{(1)}||)
\begin{pmatrix}
0 & 0\cr
0 & 1 
\end{pmatrix}}.
\end{eqnarray*}
\item ({\bf Unitarization}) As the columns of $M_2$ form an orthogonal basis and as $det(M_2)>0$, one can write
\begin{equation*}
M_2=\begin{pmatrix}
 a_{11}^{(2)} & a_{12}^{(2)}\\
 a_{21}^{(2)} & a_{22}^{(2)} 
\end{pmatrix}=
\begin{pmatrix}
\cos(t_2) & -\sin(t_2)\cr
\sin(t_2) & \cos(t_2) 
\end{pmatrix}
=e^{t_2\begin{pmatrix}
0 & 1\cr
-1 & 0 
\end{pmatrix}},
\end{equation*}
and as $M_2$ is in a neighbourhood of $I_2$, one has
$t_2=\arctan(\frac{a_{21}}{a_{11}})$.
\item ({\bf Summary})
\begin{equation*}
MTD=M_2=e^{\arctan(\frac{a_{21}}{a_{11}})
\begin{pmatrix}
0 & 1\cr
-1 & 0 
\end{pmatrix}},
\end{equation*}
hence
\begin{eqnarray*}
M&=&e^{\arctan(\frac{a_{21}}{a_{11}})
\tiny{\begin{pmatrix}
0 & 1\cr
-1 & 0 
\end{pmatrix}}}D^{-1}T^{-1}\cr
&&\cr
&=&e^{\arctan(\frac{a_{21}}{a_{11}})
\tiny{\begin{pmatrix}
0 & 1\cr
-1 & 0 
\end{pmatrix}}}
e^{\log(||C_1||)
\tiny{\begin{pmatrix}
1 & 0\cr
0 & 0 
\end{pmatrix}}
}
e^{
\log(||C_2^{(1)}||)
\tiny{\begin{pmatrix}
0 & 0\cr
0 & 1 
\end{pmatrix}}
}
e^{\frac{\scal{C_1}{C_2}}{||C_1||^2}
\tiny{\begin{pmatrix}
0 & 1\cr
0 & 0 
\end{pmatrix}}
}.
\end{eqnarray*}
\end{enumerate}
One then gets a Wei-Norman decomposition of $M$ with respect to the basis of the Lie algebra $\frak{gl}(2,\R)$: $\begin{pmatrix}
0 & 1\cr
-1 & 0 
\end{pmatrix},
\begin{pmatrix}
1 & 0\cr
0 & 0 
\end{pmatrix},
\begin{pmatrix}
0 & 0\cr
0 & 1 
\end{pmatrix},
\begin{pmatrix}
0 & 1\cr
0 & 0 
\end{pmatrix}$.
\end{example}

Now, in infinite dimensions, {\it i.e.} here within the algebra of double series (whose support is a subset of $Y^*\otimes Y^*$) endowed with the law $\shuffle\hat\otimes \bullet $,  we have Sch\"utzenberger's factorization(s) \cite{DT,RT} as a perfect analogue of Wei-Norman's theorem
for the group of group-like series. For $\mathcal{D}_{\minishuffle}$
\begin{eqnarray*}
\mathcal{D}_{\minishuffle}=\Prod_{l\in\Lyn Y}^{\searrow}\exp(S_l\otimes P_l)\in\calH_{\minishuffle}^{\vee}\hat\otimes\calH_{\minishuffle};
\end{eqnarray*}
or with the law $\stuffle\hat\otimes \bullet $, we also have the extension of Sch\"utzenberger's factorization for $\mathcal{D}_{\ministuffle}$ which is then \cite{BDM,acta,VJM}
\begin{eqnarray*}
\mathcal{D}_{\ministuffle}=\prod_{l\in\Lyn Y}^{\searrow}\exp(\Sigma_l\otimes\Pi_l)\in\calH_{\ministuffle}^{\vee}\hat\otimes\calH_{\ministuffle}.
\end{eqnarray*}
These can be used to provide a system of local coordinates on the Hausdorff group ({\it i.e.} group of group-like elements\footnote{In fact, these series are respectively characters for $\shuffle$ or $\stuffle$.}). Applying these factorizations to the multiple zeta functions $\zeta_{\shuffle}, \zeta_{\stuffle}$, or to $Z_{\shuffle}$ and $Z_{\stuffle}$ (which are all group-like), we have the representations
\begin{eqnarray*}
Z_{\shuffle}=\prod_{l\in\Lyn X\mysetminusD X}^{\searrow}e^{\zeta(S_l)P_l}&\text{and}&Z_{\stuffle}=\prod_{l\in\Lyn Y\mysetminusD \{y_1\}}^{\searrow}e^{\zeta(\Sigma_l)\Pi_l}.
\end{eqnarray*}
 It means that all relations among polyzetas which can be seen here will be taken from relations among their local coordinates. Our method is to use identity \eqref{regularization2} to reduce relations between the two systems of local coordinates $\{\zeta(S_l)\}_{l\in \Lyn X}$ and $\{\zeta(\Sigma_l)\}_{l\in \Lyn Y}$.

\section{Structure of polyzetas}\label{Structure}
\subsection{Representations of polynomials on bases}\label{representation}
The aim of this subsection is to provide a method to represent any polynomial of $\Q\langle Y\rangle$ in terms of each basis $\{P_w\}_{w\in Y^*}$, $\{S_w\}_{w\in Y^*}$, $\{\Pi_w\}_{w\in Y^*}$ or $\{\Sigma_w\}_{w\in Y^*}$. 

Recall that the bases $\{P_w\}_{w\in Y^*}$ and $\{\Pi_w\}_{w\in Y^*}$ are homogeneous and upper triangular, the bases $\{S_w\}_{w\in Y^*}$ and $\{\Sigma_w\}_{w\in Y^*}$ are homogeneous and lower triangular\footnote{w.r.t the words and the lexicographic ordering, for example, $\Sigma_w=w+\sum_{v<w,(v)=(w)}\alpha_vv.$}. 
Without loss of generality we can assume that $P\in \Q\langle Y\rangle$ is a homogeneous polynomial of weight $n$, we now represent $P$ in terms of the basis $\{\Sigma_w\}_{w\in Y^*}$ by the following algorithm.
\subsubsection*{Algorithm $1$}
\textbf{INPUT:} A homogeneous polynomial $P$ of weight $n$.

\textbf{OUTPUT:} The representation of $P$ in terms of the basis $\{\Sigma_w\}_{w\in Y^*}$.
\begin{itemize}
\item[\textbf{Step 1.}]  We choose the leading term\footnote{This term includes the greatest word in the support of $P$ and its coefficient.} of $P$, assumed $\lambda_1w_1$. Expressing the word $w_1$ as follows
\begin{equation}
 w_1=\Sigma_{w_1}+\sum_{v<w_1, (v)=n}{\alpha_vv}.
\end{equation}
The polynomial $P$ can now be rewritten in the form 
\begin{equation}
P=\lambda_{w_1}\Sigma_{w_1}+\sum_{v<w_1, (v)=n}{\beta_vv}.\label{express2}
\end{equation} 
\item[\textbf{Step 2.}] We repeat \textbf{Step 1} with $P$ now understood as the polynomial $\sum_{v<w_1, (v)=n}{\beta_vv}$, and so on until the last monomial which admits the smallest word of weight $n$, $y_n$, and we really have $y_n=\Sigma_{y_n}$. At last, by 
re-expressing the coefficients, we will obtain the representation of the original in form that
\begin{equation}
 P=\sum_{v\leq w_1, (v)=n}{\lambda_v\Sigma_v}.
\end{equation}
\end{itemize}

\begin{example}
 $P:=2y_1y_2-1/2y_3.$\\
 Step 1. Since $\Sigma_{y_1y_2}=y_1y_2+y_2y_1+y_3$, we replace $y_1y_2$ with $\Sigma_{y_1y_2}-y_2y_1-y_3$ in $P$
 $$P=2\Sigma_{y_1y_2}-2y_2y_1-5/2y_3.$$
 Step 2. Since $\Sigma_{y_2y_1}=y_2y_1+1/2y_3$, we replace $y_2y_1$ with $\Sigma_{y_2y_1}-1/2y_3$ in $P$
  $$P=2\Sigma_{y_1y_2}-2\Sigma_{y_2y_1}-3/2y_3.$$
Since $y_3=\Sigma_{y_3} $, we thus get $P=2\Sigma_{y_1y_2}-2\Sigma_{y_2y_1}-3/2\Sigma_{y_3}.$
\end{example}

\begin{corollary}
For any $w\in Y^*$, we can represent\footnote{$|w|$ and $(w)$ respectively denote the length and the weight of the word $w$.}
\begin{eqnarray*}
w=&P_w+\Sum_{u>w, |u|=|w|}\alpha_u^1P_u&= S_w+\sum_{u<w,|u|=|w|}\alpha_u^2S_u,\\
w=&\Pi_w+\Sum_{v>w, (v)=(w)}\beta_v^1\Pi_v&= \Sigma_w+\sum_{v<w,(v)=(w)}\beta_v^2\Sigma_v.
\end{eqnarray*}
\end{corollary}

\subsection{Identifying the local coordinates}
We now use the alphabet $X=\{x_0,x_1\}$ ordered by $x_0<x_1$. Returning to formula \eqref{regularization2}, with the bases $\{P_w\}_{w\in X^*}$ and $\{S_w\}_{w\in X^*}$ defined as \eqref{basisP} and \eqref{basisS}, we will find relations among polyzetas by identifying on the bases as local coordinates.
First, we expand $B'$, given in \eqref{regularization2}, in form of generating series of $y_1$.
\begin{lemma}\label{expandB'1}
We have
\begin{eqnarray*}
B'(y_1)=1+\sum_{m\ge2}{B^{(m)}y_1^m},&\mbox{with}&
B^{(m)}=\sum_{i=1}^{\lfloor{m}/{2}\rfloor}\sum_{k_1,\ldots,k_i\ge2\atop k_1+\ldots+k_i=m}(-1)^{m-i}\frac{\zeta(k_1)\ldots\zeta(k_i)}{k_1\ldots k_i},
\end{eqnarray*}
     where $\lfloor{m}/{2}\rfloor$ is the largest integer not greater than ${m}/{2}$.
   \end{lemma}

\begin{proof}
Expanding the exponential, one has successively
\begin{eqnarray*}
B'(y_1)&=& \sum_{n\ge0} \frac{1}{n!}\left(\sum_{k\geq2}{\frac{(-1)^{k-1}\zeta (k)}{k}y_1^k}\right)^n\\
&=&\sum_{n\ge0}\frac{1}{n!}\sum_{k_1,\ldots,k_n\geq2}{\frac{(-1)^{k_1+\ldots+k_n-n}\zeta(k_1)\ldots\zeta(k_n)}{k_1\ldots k_n}y_1^{k_1+\ldots+k_n}}\\
&=&1+\sum_{m\ge2}\left(\sum_{n=1}^{\lfloor{m}/{2}\rfloor} \frac{1}{n!}\sum_{k_1,\ldots,k_n\geq2\atop k_1+\ldots+k_n=m}\frac{(-1)^{m-n}\zeta(k_1)\ldots\zeta(k_n)}{k_1\ldots k_n}\right) y_1^{m}\\
&=&1+\sum_{m\ge2}{B^{(m)}y_1^m}.
\end{eqnarray*}
\end{proof}

\begin{example}  
$$\begin{array}{rcl}
B^{(2)}&=&-\frac{\zeta(2)}{2},\\
B^{(3)}&=&\frac{\zeta(3)}{3},\\
B^{(4)}&=&-\frac{\zeta(4)}{4}+\frac{\zeta(2)^2}{2^2},\\
B^{(5)}&=&\frac{\zeta(5)}{5}-2\frac{\zeta(2)}{2}\frac{\zeta(3)}{3}.
\end{array}$$
\end{example}

\subsubsection{Identifying with respect to the basis $\{\Pi_w\}_{w\in Y^*}$}
Using the duality of the bases, we rewrite \eqref{regularization2} as follows
\begin{eqnarray}
\sum_{v\in Y^*}{\zeta_{\ministuffle}(\Sigma_{v})\Pi_{v}}=B'(y_1)\sum_{v\in Y^*}\zeta_{\shuffle}(\pi_X(\Sigma_{v})) \Pi_{v}.\label{identify1}
\end{eqnarray}
Moreover, we see that $B'(y_1)$ is a series of a single letter (like a single variable), $y_1$, and 
\begin{eqnarray*}
y_1^k\Pi_v=\Pi_{y_1}^k\Pi_v=\Pi_{y_1^kv},&&\forall k\ge1, v\in Y^*.
\end{eqnarray*}
We can then identify the coefficients in \eqref{identify1} and obtain:

\begin{proposition}\label{formulsurPi}
\begin{itemize}
\item[i)] For any $v\in Y^*\mysetminusD\ y_1Y^*$, one has\footnote{As $x_0X^*x_1$ and $Y^*\mysetminusD\ y_1Y^*$ are disjointed, the unique notation
$\zeta(P)$ is used here to replace $\zeta_{\shuffle}(P)$ or $\zeta_{\ministuffle}(P)$
if the polynomial $P$ only contains convergent words.}
$\zeta(\Sigma_{v})=\zeta(\pi_X\Sigma_{v})$.

\item[ii)] For any $ v=y_1^kw\in Y^*, k\ge1, w\in Y^*\mysetminusD\  y_1Y^*$, one has
\begin{eqnarray*}
\zeta_{\shuffle}(\pi_X\Sigma_{v})+\sum_{m=2}^k{B^{(m)}\zeta_{\shuffle}(\pi_X\Sigma_{y_1^{k-m}w})}=0.
\end{eqnarray*}
\end{itemize}
\end{proposition}

\begin{proof}
From Lemma \ref{expandB'1}, we see that $\scal{B'(y_1)}{y_1^0}=1,\scal{B'(y_1)}{y_1}=0$ and 
\begin{eqnarray*}
\forall m\ge2,&&\scal{B'(y_1)}{y_1^m}=B^{(m)}.
\end{eqnarray*}
Using the basis $\{\Pi_w\}_{w\in Y^*}$ as a coordinate system, we identify the coefficients of the two sides in \eqref{identify1} and obtain the preceding statements.
\end{proof}

\begin{example}
\begin{enumerate}  
\item For $v=y_2,\zeta(\Sigma_{y_2}) = \zeta(S_{x_0x_1})$.
\item For $v=y_2y_3,\zeta(\Sigma_{y_2y_3})=\zeta(S_{x_0x_1x_0^2x_1})-2\zeta(S_{x_0^2x_1x_0x_1})-2\zeta(S_{x_0^3x_1^2})
+\zeta(S_{x_0^4x_1})$.
\item For $v=y_1^3,-\frac12\zeta(S_{x_0x_1^2})+\frac16\zeta(S_{x_0^2x_1})+B^{(3)}=0$.
\item For $v=y_1^2y_2,\zeta(S_{x_0x_1^3})-\zeta(S_{x_0^2x_1^2})+\frac12\zeta(S_{x_0^3x_1})+B^{(2)}=0$.
\end{enumerate}  
\end{example}

\subsubsection{Identifying with respect to the basis $\{P_w\}_{w\in X^*}$}
Let us denote by\footnote{They are defined by $P'_w=\pi_X(\pi_YP_w), \ \forall w\in X^*$. Note that $\pi_YP_w=\pi_Yw=0,\ \forall w\in X^*x_0$.} $\{P'_w\}_{w\in X^*x_1}$ the reductions of $\{P_w\}_{w\in X^*x_1}$ on $\Q\oplus \Q\langle X\rangle x_1$. By applying the mapping $\pi_X$ on the two sides of \eqref{identify1} and using the duality of the bases, we can rewrite the regularization as follows
\begin{eqnarray}
B'(x_1)\sum_{u\in X^*x_1}\zeta_{\shuffle}(S_{u})P'_{u}=\sum_{u\in X^*x_1}{\zeta_{\ministuffle}(\pi_YS_{u}) P'_{u}}.\label{identify2}
\end{eqnarray}

Similarly, remarking that $B'(x_1)$ is a series of a single letter, $x_1$,
$$x_1^kP_u=P_{x_1}^kP_u=P_{x_1^ku},\quad
\forall k\ge1, u\in X^*.$$   

\begin{proposition}\label{formulsurP}
\begin{itemize}
\item[i)] For any $u\in X^*\mysetminusD\  x_1X^*,\ \zeta(S_{u})=\zeta(\pi_YS_{u})$.

\item[ii)] For any $u\in x_1X^*\mysetminusD\  x_1^2X^*,\ \zeta_{\stuffle}(\pi_YS_u)=0$.

\item[iii)] For any $ u=x_1^kw\in X^*, k\ge2, w\in X^*\mysetminusD\  x_1X^*,\ B^{(k)}\zeta(S_w)=\zeta_{\ministuffle}(\pi_YS_{u})$.
\end{itemize}
\end{proposition}  
\begin{proof}
Similarly to Proposition \ref{formulsurPi}, admitting the basis $\{P_w\}_{w\in X^*}$ as a coordinate system, we identify the coefficients of the two sides in \eqref{identify2} and then obtain the statements.
\end{proof}

\begin{example}
\begin{enumerate}  
\item For $u=x_0x_1,\zeta(S_{x_0x_1})=\zeta(\Sigma_{y_2})$.
\item For $u=x_0x_1x_0^2x_1,\zeta(S_{x_0x_1x_0^2x_1}) = \zeta(\Sigma_{y_2y_3})+2\zeta(\Sigma_{y_3x_2})+6\zeta(\Sigma_{y_4x_1})-5\zeta(\Sigma_{y_5})$.
\item For $u=x_1x_0x_1,\zeta(\Sigma_{y_2y_1})-\frac32\zeta(\Sigma_{y_3})=0$.
\item For $u=x_1^2x_0x_1,B^{(2)}\zeta(S_{x_0x_1})=2\zeta(\Sigma_{y_4})-\zeta(\Sigma_{y_2})^2-\zeta(\Sigma_{y_3y_1})$.
\end{enumerate}
\end{example}
  
\subsection{Algorithms to represent the structure of polyzetas}
From Proposition \ref{formulsurPi} and \ref{formulsurP}, we really have relations among polyzetas represented on the bases $\{S_w\}_{w\in X^*}$ and $\{\Sigma_w\}_{w\in Y^*}$.

In fact, thanks to the formulas \eqref{basisS} and \eqref{basisSigma}, we can easily represent these relations on the pure transcendence bases $\{S_l\}_{l\in \Lyn X}$ or $\{\Sigma_l\}_{l\in \Lyn Y}$ respectively.

In the two following algorithms, one uses these relations and the other one (\textit{Algorithm} $3$) uses as well the structures of shuffle 
and stuffle products, we will eliminate these relations, in weight, to find the structure of polyzetas represented on the bases $\{S_l\}_{l\in \Lyn X}$ and $\{\Sigma_l\}_{l\in \Lyn Y}$. The following two algorithms will be proceeded by recurrence on the weight of the words.

The same result obtained will be shown in the next subsection.

 \subsubsection*{Algorithm $2$} This algorithm uses Proposition \ref{formulsurPi} and \textit{Algorithm $1$} to establish polynomial  relations among polyzetas on the basis $\{S_l\}_{l\in \Lyn X}$ or uses Proposition \ref{formulsurP} and \textit{Algorithm} $1$ to establish relations among polyzetas on the basis $\{\Sigma_l\}_{l\in \Lyn Y}$.

We display here the second case.
 
 \textbf{INPUT:} A positive integer $n$.
 
 \textbf{OUTPUT:} The representations of polyzetas of weight $n$ in terms of irreducible elements of polyzetas on the transcendence basis $\{\Sigma_l\}_{l\in \Lyn Y}$.
\begin{itemize}
\item[\textbf{Step 1.}] We set the list, denoted by $X_n$, of all words\footnote{Note that, there are $2^{n-1}$ words of weight $n$.} of weight\footnote{In the alphabet $X$, the weight of a word is understood as the length of that word.} $n$ of $X^*x_1$.
 \item[\textbf{Step 2.}] For each $w\in X_n$, we set the polynomial  $\mathcal{P}:=\pi_Y(S_w)$ in $\QY$ and thanks to \textit{Algorithm} $1$ we represent $\zeta(\mathcal{P})$ in terms of $\{\zeta(\Sigma_l)\}_{l\in \Lyn Y}$. By taking the representations of $\zeta(\Sigma_l)$'s from the data of lower weights, we make representation in terms of irreducible elements for $\zeta(\mathcal{P})$ and proceed to establish a polynomial relation as follows:
\begin{itemize}
 \item[i)] If $w\in \Lyn X$ then we store $\zeta(\mathcal{P})$ to the variable $\zeta(S_w)$,
 \item[ii)] If $w=x_1u, u\in x_0X^*x_1$ then we make the relation $\zeta(\mathcal{P})=0$.
 \item[iii)] If $w\in x_0X^*x_1\mysetminusD  \Lyn X$, we rewrite $w$ in the form of Lyndon factorization, $w=l_1^{i_1}\ldots l_k^{i_k}$. By taking $\zeta(S_{l_j}),\ j=1\ldots k$ from the data of lower weights, we make the relation 
 $$\frac1{i_1!\ldots i_k!}\zeta(S_{l_1})^{i_1}\ldots\zeta(S_{l_k})^{i_k}=\zeta(\mathcal{P}).$$
 \end{itemize} 
 \item[\textbf{Step 3.}] We reduce the above relations to representations of polyzetas in terms of irreducible elements.
 \end{itemize}  

The next lemma will give another way to find the relations among the family $\{\zeta_{\shuffle}(S_w)\}_{w\in X^*}$ and the family $\{\zeta_{\stuffle}(\Sigma_w)\}_{w\in Y^*}$.
\begin{lemma}\label{express}
\begin{enumerate}
\item[i)] For any $l_1,l_2 \in\Lyn X\mysetminusD  X$ (resp. $l_1, l_2\in \Lyn Y\mysetminusD \{y_1\}$), one has
\begin{eqnarray*}
\zeta(S_{l_1}\shuffle S_{l_2})&=&\zeta(\pi_Y(S_{l_1})\stuffle\pi_Y(S_{l_2})),\\
\zeta(\Sigma_{l_1}\ministuffle \Sigma_{l_2})&=&\zeta(\pi_X(\Sigma_{l_1})\shuffle\pi_X(\Sigma_{l_2}))).
\end{eqnarray*}
\item[ii)] For any $w\in x_0X^*x_1$ or $w\in x_1x_0X^*x_1$ (resp. $w\in Y^*\mysetminusD\  y_1^2Y^*$), one has
\begin{eqnarray*}
\zeta_{\shuffle}(S_w)&=&\zeta_{\stuffle}(\pi_Y(S_w)),\\
\zeta_{\ministuffle}(\Sigma_w)&=&\zeta_{\shuffle}(\pi_X(\Sigma_w)).
\end{eqnarray*}
\end{enumerate}
\end{lemma}

\begin{proof}
Remark that, for any $w\in X^*,S_w=w+\sum_{v<w}\alpha_vv$ and if $l \in\Lyn X\mysetminusD  X$ then $l\in x_0X^*x_1$.

Relying on properties of polyzetas on words, {\it i.e.} \cite{acta,VJM}
\begin{eqnarray*}
\zeta(l_1\shuffle l_2)&=&\zeta(\pi_Y(l_1)\stuffle \pi_Y(l_2)),\quad \forall l_1,l_2 \in\Lyn X\mysetminusD  X,\\
\zeta_{\shuffle}(x_1\shuffle l)&=&\zeta_{\stuffle}(y_1\stuffle\pi_Y(l)),\quad \forall l \in\Lyn X\mysetminusD  X,
\end{eqnarray*}
we get the expected results.
\end{proof} 

\begin{example}
For $l_1=x_0x_1, l_2=x_0^2x_1^2$ (in $\Lyn X$) and $l_1=y_2, l_2=y_3y_1$ (in $\Lyn Y$):
\begin{eqnarray*}
\zeta(S_{x_0x_1})\zeta(S_{x_0^2x_1^2})&=&\zeta(\Sigma_{y_2})\zeta(\Sigma_{y_3y_1})-\frac{1}{2}\zeta(\Sigma_{y_2})\zeta(\Sigma_{y_4}),\\
\zeta(\Sigma_{y_2})\zeta(\Sigma_{y_3y_1})&=&\zeta(S_{x_0x_1})\zeta(S_{x_0^2x_1^2})+\frac{1}{2}\zeta(S_{x_0x_1})\zeta(S_{x_0^3x_1}).
\end{eqnarray*}
For $w=x_1x_0^2x_1$ (in $x_1x_0X^*x_1$) and $w=y_1y_3$ (in $y_1Y^*$):
\begin{eqnarray*}
0&=&\frac12\zeta(\Sigma_{y_2})^2+\zeta(\Sigma_{y^3y_1})-2\zeta(\Sigma_{y_4}),\\
0&=&-\frac12\zeta(S_{x_0x_1})^2+\zeta(S_{x_0^2x_1^2})+\zeta(S_{x_0^3x_1}).
\end{eqnarray*}
\end{example}
 
\subsubsection*{Algorithm $3$} This algorithm uses Lemma \ref{express} and \textit{Algorithm $1$}
to establish polynomial relations among polyzetas on the basis $\{S_l\}_{l\in \Lyn X}$ or the basis $\{\Sigma_l\}_{l\in \Lyn Y}$.

We display here the second case.

 \textbf{INPUT:} A positive integer $n$.
 
 \textbf{OUTPUT:} The representations of polyzetas of weight $n$ in terms of irreducible elements of polyzetas on the transcendence basis $\{\Sigma_l\}_{l\in \Lyn Y}$.
 \begin{itemize}
 \item[\textbf{Step 1.}] We set a list, denoted by $X_n$, all words of weight $n$ in $x_0X^*x_1$ or $x_1x_0X^*x_1$.
 \item[\textbf{Step 2.}] We establish polynomial relations of weight $n$ as follows. For each $w\in X_n$, we make a polynomial $\mathcal{P}$ in $\Q\langle Y\rangle$ by the way:
 \begin{itemize}
 \item[i)] If $w\in \Lyn X$ then $\mathcal{P}:=\pi_Y(S_{l_1})\stuffle \pi_Y(S_{l_2})-\pi_Y(S_{l_1}\shuffle S_{l_2})$, where $(l_1,l_2)$ is the standard factorization of $w$.
 \item[ii)] If $w=x_1w_1$ then $\mathcal{P}:=\pi_Y(S_{x_1})\stuffle \pi_Y(S_{w_1})-\pi_Y(S_{x_1}\shuffle S_{w_1})$.
 \item[iii)] If $w=l_l^{i_1}\ldots l_k^{i_k},\ l_1,\ldots,l_k\in \Lyn X,\ l_1>\ldots>l_k$ then \\
 $\mathcal{P}:=\pi_Y(S_{l_1})^{\stuffle i_1}\stuffle \ldots \stuffle \pi_Y(S_{l_k})^{\stuffle i_k}-\pi_Y(S_{l_1}\shuffle \ldots\shuffle S_{l_k})$.
 \end{itemize} 
Thanks to \textit{Algorithm} $1$, we represent $\zeta(\Sigma_{\mathcal{P}})$ in terms
of $\{\zeta(\Sigma_l)\}_{l\in \Lyn Y}$ (here, $\zeta(\Sigma_l)$ are taken from the data of lower weights).
At last, we make the relation $\zeta(\Sigma_{\mathcal{P}})=0$.
 \item[\textbf{Step 3.}] We reduce the above relations to representations of polyzetas in terms of irreducible elements.
\end{itemize}
 
These algorithms produce homogeneous polynomial relations among local coordinates
$\{\zeta(\Sigma_l)\}_{l\in\Lyn Y}$ (resp. $\{\zeta(S_l)\}_{l\in\Lyn X}$).
Each identity is indexed by a Lyndon word and is not an identity of the tautological form
\begin{eqnarray}\label{tautologies}
\zeta(\Sigma_l)=\zeta(\Sigma_l)&(\mbox{or }\zeta(S_l)=\zeta(S_l)).
\end{eqnarray} 

Replacing "$=$'' by "$\longrightarrow$''  in these homogeneous polynomial relations,
we obtain a noetherian rewriting system among $\{\zeta(\Sigma_l)\}_{l\in\Lyn Y}$
(resp. $\{\zeta(S_l)\}_{l\in\Lyn X}$) in which irreducible terms are polyzetas
involved in tautologies \eqref{tautologies} and they are viewed as algebraic generators
of the algebra of convergent polyzetas \cite{acta,VJM}.

\subsection{Results}\label{resuls}
\subsubsection{Representation of polyzetas in terms of irreducible polyzetas}
The following results were computed by our package in Maple \cite{program} thanks to \textit{Algorithm} $2$ (or \textit{Algorithm} $3$). 
 We show here representations of polyzetas in terms of irreducible polyzetas of the bases indexed by Lyndon words on the two alphabets $X$ and $Y$.

For each weight $n$, the list of Lyndon words $l\in\Lyn Y$ will be displayed in the second column,
and their projection over $X$, \textit{i.e.} $\pi_X(l)\in\Lyn X$, will be displayed in the fourth column
which are also, due to a lemma by D. Perrin, the list of Lyndon words in $\Lyn Y$ (see Table $1$).
$$\begin{array}{|c||c|c||c|c|}
\hline
n&l & \zeta(\Sigma_{l}) & \pi_X(l) & \zeta(S_{\pi_X(l)})\\ \hline \hline
3 & y_2y_1 & \frac{3}{2}\zeta(\Sigma_{y_3}) & x_0x_1^2 &\zeta(S_{x_0^2x_1})\\
\hline
& y_4 & \frac{2}{5}\zeta(\Sigma_{y_2})^2 & x_0^3x_1 & \frac{2}{5}\zeta(S_{x_0x_1})^2\\              
4 &y_3y_1& \frac{3}{10}\zeta(\Sigma_{y_2})^2 & x_0^2x_1^2& \frac{1}{10}\zeta(S_{x_0x_1})^2 \\ 
&y_2y_1^2& \frac{2}{3}\zeta(\Sigma_{y_2})^2 & x_0x_1^3& \frac{2}{5}\zeta(S_{x_0x_1})^2\\ 
\hline
& y_4y_1 & -\zeta(\Sigma_{y_3})\zeta(\Sigma_{y_2})+\frac{5}{2}\zeta(\Sigma_{y_5}) & 
x_0^3x_1^2 & -\zeta(S_{x_0^2x_1})\zeta(S_{x_0x_1})+2\zeta(S_{x_0^4x_1})\\ 
& y_3y_2 & 3\zeta(\Sigma_{y_3})\zeta(\Sigma_{y_2})-5\zeta(\Sigma_{y_5}) &
x_0^2x_1x_0x_1 & -\frac{3}{2}\zeta(S_{x_0^4x_1})+\zeta(S_{x_0^2x_1})\zeta(S_{x_0x_1})\\
5 &y_3y_1^2& \frac{5}{12}\zeta(\Sigma_{y_5}) & x_0^2x_1^3 & -\zeta(S_{x_0^2x_1})\zeta(S_{x_0x_1})+2\zeta(S_{x_0^4x_1}) \\
&y_2^2y_1& \frac{3}{2}\zeta(\Sigma_{y_3})\zeta(\Sigma_{y_2})-\frac{25}{12}\zeta(\Sigma_{y_5}) & x_0x_1x_0x_1^2 & \frac{1}{2}\zeta(S_{x_0^4x_1})\\&
y_2y_1^3& \frac{1}{4}\zeta(\Sigma_{y_3})\zeta(\Sigma_{y_2})+\frac{5}{4}\zeta(\Sigma_{y_5}) & x_0x_1^4 & \zeta(S_{x_0^4x_1})\\
 \hline
& y_6 & \frac{8}{35}\zeta(\Sigma_{y_2})^3 & 
 x_0^5x_1 & \frac{8}{35}\zeta(S_{x_0x_1})^3\\&
y_5y_1 & \frac{2}{7}\zeta(\Sigma_{y_2})^3-\frac{1}{2}\zeta(\Sigma_{y_3})^2  & 
x_0^4x_1^2 & \frac{6}{35}\zeta(S_{x_0x_1})^3-\frac{1}{2}\zeta(S_{x_0^2x_1})^2\\&
y_4y_2 & \zeta(\Sigma_{y_3})^2-\frac{4}{21}\zeta(\Sigma_{y_2})^3 &
x_0^3x_1x_0x_1 & \frac{4}{105}\zeta(S_{x_0x_1})^3\\&
 y_4y_1^2 & \frac{3}{10}\zeta(\Sigma_{y_2})^3-\frac{3}{4}\zeta(\Sigma_{y_3})^2 & 
x_0^3x_1^3 & \frac{23}{70}\zeta(S_{x_0x_1})^3-\zeta(S_{x_0^2x_1})^2\\ 
6 & y_3y_2y_1 & 3\zeta(\Sigma_{y_3})^2-\frac{9}{10}\zeta(\Sigma_{y_2})^3 & 
x_0^2x_1x_0x_1^2 & \frac{2}{105}\zeta(S_{x_0x_1})^3\\ 
 &y_3y_1y_2 & -\frac{17}{30}\zeta(\Sigma_{y_2})^3+\frac{9}{4}\zeta(\Sigma_{y_3})^2
 & x_0^2x_1^2x_0x_1 & -\frac{89}{210}\zeta(S_{x_0x_1})^3+\frac{3}{2}\zeta(S_{x_0^2x_1})^2 \\ &
y_3y_1^3 & \frac{1}{21}\zeta(\Sigma_{y_2})^3 & 
x_0^2x_1^4 & \frac{6}{35}\zeta(S_{x_0x_1})^3-\frac{1}{2}\zeta(S_{x_0^2x_1})^2 \\ &
y_2^2y_1^2 & \frac{11}{63}\zeta(\Sigma_{y_2})^3-\frac{1}{4}\zeta(\Sigma_{y_3})^2
& 
x_0x_1x_0x_1^3 & \frac{8}{21}\zeta(S_{x_0x_1})^3-\zeta(S_{x_0^2x_1})^2 \\ &
y_2y_1^4 & \frac{17}{50}\zeta(\Sigma_{y_2})^3+\frac{3}{16}\zeta(\Sigma_{y_3})^2 & 
x_0x_1^5 & \frac{8}{35}\zeta(S_{x_0x_1})^3\\ 
\hline
\end{array}$$
\centerline{\qquad Table $1$: Representation of polyzetas in terms of irreducible polyzetas up to weight $6$.}
  
\subsubsection{Conclusion of the results}
Let us denote by ${\cal Z}_n$ the $\Q$-vector space generated by polyzetas of weight $n$ and $d_n$ its dimension.

From the above representations, we obtain their bases as follows:
\begin{itemize}
\item $n=2,d_2=1,$\ \ \ ${\cal Z}_2=\mathrm{span}_{\Q}\{\zeta(\Sigma_{y_2})\}              \ \ \ =\mathrm{span}_{\Q}\{\zeta(S_{x_0x_1})\}
$
\item $n=3,d_3=1,$\ \ \   
$ {\cal Z}_3=\mathrm{span}_{\Q}\{\zeta(\Sigma_{y_3})\}              \ \ \  =\mathrm{span}_{\Q}\{\zeta(S_{x_0^2x_1})\}
$
\item $n=4,d_4=1,$ \ \ \ 
$ {\cal Z}_4=\mathrm{span}_{\Q}\{\zeta(\Sigma_{y_2})^2\}
=\mathrm{span}_{\Q}\{\zeta(S_{x_0x_1})^2\} 
 $
\item $n=5,d_5=2,$ \begin{eqnarray*} {\cal Z}_5=\mathrm{span}_{\Q}\{\zeta(\Sigma_{y_5}),  \zeta(\Sigma_{y_2})\zeta(\Sigma_{y_3})\} =\mathrm{span}_{\Q}\{\zeta(S_{x_0^4x_1}), \zeta(S_{x_0x_1})\zeta(S_{x_0^2x_1})\}
 \end{eqnarray*}  
\item $n=6,d_6=2,$ \begin{eqnarray*} {\cal Z}_6&=&\mathrm{span}_{\Q}\{\zeta(\Sigma_{y_2})^3,\zeta(\Sigma_{y_3})^2\} 
=\mathrm{span}_{\Q}\{\zeta(S_{x_0x_1})^3, \zeta(S_{x_0^2x_1})^2\}
 \end{eqnarray*}  
\item $n=7,d_7=3,$ \begin{eqnarray*} {\cal Z}_7&=&\mathrm{span}_{\Q}\{\zeta(\Sigma_{y_7}), \zeta(\Sigma_{y_2})\zeta(\Sigma_{y_5}), \zeta(\Sigma_{y_2})^2\zeta(\Sigma_{y_3})\} \\       &=&\mathrm{span}_{\Q}\{\zeta(S_{x_0^6x_1}), \zeta(S_{x_0x_1})\zeta(S_{x_0^4x_1}),\zeta(S_{x_0x_1})^2\zeta(S_{x_0^2x_1})\}
 \end{eqnarray*}  
\item $ n=8,d_8=4,$ \begin{eqnarray*} {\cal Z}_8&=&\mathrm{span}_{\Q}\{\zeta(\Sigma_{y_2})^4, \zeta(\Sigma_{y_3y_1^5}), \zeta(\Sigma_{y_3})\zeta(\Sigma_{y_5}),  \zeta(\Sigma_{y_3})^2\zeta(\Sigma_{y_2})\} \\       
 &=&\mathrm{span}_{\Q}\{\zeta(S_{x_0x_1})^4, \zeta(S_{x_0^2x_1})\zeta(S_{x_0^4x_1}),\zeta(S_{x_0x_1})\zeta(S_{x_0^2x_1})^2,\\
 &&\zeta(S_{x_0x_1^2x_0x_1^4})\}
 \end{eqnarray*}  
\item $n=9,d_9=5,$ 
\begin{eqnarray*} {\cal Z}_9&=&\mathrm{span}_{\Q}\{\zeta(\Sigma_{y_9}), \zeta(\Sigma_{y_2})^2\zeta(\Sigma_{y_5}), \zeta(\Sigma_{y_2})\zeta(\Sigma_{y_7}),\zeta(\Sigma_{y_2})^3\zeta(\Sigma_{y_3}),\\
&& \zeta(\Sigma_{y_3})^3\} \\       &=&\mathrm{span}_{\Q}\{\zeta(S_{x_0^8x_1}), \zeta(S_{x_0x_1})^2\zeta(S_{x_0^4x_1}), \zeta(S_{x_0^2x_1})^3, \zeta(S_{x_0x_1})\zeta(S_{x_0^6x_1}),\\
 && \zeta(S_{x_0x_1})^3\zeta(S_{x_0^2x_1})\}
 \end{eqnarray*}  
\item $n=10,d_{10}=7,$ \begin{eqnarray*} {\cal Z}_{10}&=&\mathrm{span}_{\Q}\{\zeta(\Sigma_{y_2})^5, \zeta(\Sigma_{y_5})^2, \zeta(\Sigma_{y_3y_1^7}),\zeta(\Sigma_{y_2})\zeta(\Sigma_{y_3})\zeta(\Sigma_{y_5}) ,\\ 
&& \zeta(\Sigma_{y_2})^2\zeta(\Sigma_{y_3})^2,\zeta(\Sigma_{y_3})\zeta(\Sigma_{y_7}),\zeta(\Sigma_{y_2})\zeta(\Sigma_{y_3y_1^5})\}\\
 &=& \mathrm{span}_{\Q}\{\zeta(S_{x_0^4x_1})^2, \zeta(S_{x_0^4x_1})\zeta(S_{x_0^2x_1})\zeta(S_{x_0x_1}),\zeta(S_{x_0x_1})^2\zeta(S_{x_0^2x_1})^2,\\
 &&\zeta(S_{x_0x_1})^5,\zeta(S_{x_0x_1^3x_0x_1^5}), \zeta(S_{x_0^6x_1})\zeta(S_{x_0^2x_1}), \zeta(S_{x_0x_1})\zeta(S_{x_0x_1^2x_0x_1^4})\}
 \end{eqnarray*}  
\item $n=11,d_{11}=9,$ \begin{eqnarray*} {\cal Z}_{11}&=&\mathrm{span}_{\Q}\{\zeta(\Sigma_{y_{11}}), \zeta(\Sigma_{y_2})^2\zeta(\Sigma_{y_7}), \zeta(\Sigma_{y_2})\zeta(\Sigma_{y_9}), \zeta(\Sigma_{y_2})^3\zeta(\Sigma_{y_5}),\\
 &&\zeta(\Sigma_{y_2y_1^9}),\zeta(\Sigma_{y_3})^2\zeta(\Sigma_{y_5}),  \zeta(\Sigma_{y_2})\zeta(\Sigma_{y_3})^3, \zeta(\Sigma_{y_2})^2\zeta(\Sigma_{y_7}),\\ &&\zeta(\Sigma_{y_3})\zeta(\Sigma_{y_3y_1^5})\}\\
 &=& \mathrm{span}_{\Q}\{\zeta(S_{x_0^{10}x_1}), \zeta(S_{x_0^4x_1})\zeta(S_{x_0^2x_1})^2,\zeta(S_{x_0^2x_1})\zeta(S_{x_0x_1})^3\zeta(S_{x_0x_1}),\\
 &&\zeta(S_{x_0^2x_1})\zeta(S_{x_0x_1})^4, \zeta(S_{x_0^4x_1})\zeta(S_{x_0x_1})^3, \zeta(S_{x_0^2x_1})\zeta(S_{x_0x_1^2x_0x_1^4}),\\ 
 &&\zeta(S_{x_0x_1^2x_0x_1^2x_0x_1^4}),  \zeta(S_{x_0^6x_1})\zeta(S_{x_0x_1})^2,
 \zeta(S_{x_0^8x_1})\zeta(S_{x_0x_1})
\}
 \end{eqnarray*}  
\item $n=12,d_{12}=12,$ \begin{eqnarray*} {\cal Z}_{12}&=&\mathrm{span}_{\Q}\{\zeta(\Sigma_{y_2})^6, \zeta(\Sigma_{y_3})^4, \zeta(\Sigma_{y_2})\zeta(\Sigma_{y_5})^2, \zeta(\Sigma_{y_3}\Sigma_{y_1^9}), \zeta(\Sigma_{y_2^2}\Sigma_{y_1^8}), \\ 
 &&\zeta(\Sigma_{y_2})^3\zeta(\Sigma_{y_3})^2,\zeta(\Sigma_{y_3})\zeta(\Sigma_{y_9}),\zeta(\Sigma_{y_2})\zeta(\Sigma_{y_3}\Sigma_{y_1^7}), \zeta(\Sigma_{y_5})\zeta(\Sigma_{y_7}),\\ 
 &&\zeta(\Sigma_{y_2})\zeta(\Sigma_{y_3})\zeta(\Sigma_{y_7}),\zeta(\Sigma_{y_2})^2\zeta(\Sigma_{y_3})\zeta(\Sigma_{y_5}),  \zeta(\Sigma_{y_2})^2\zeta(\Sigma_{y_3}\Sigma_{y_1^5})\}\\
 &=&\mathrm{span}_{\Q}\{\zeta(S_{x_0x_1})^6, \zeta(S_{x_0^2x_1})^4 , \zeta(S_{x_0x_1x_0x_1^9}),\zeta(S_{x_0x_1})^2\zeta(S_{x_0x_1^2x_0x_2^4}), \\ 
 &&\zeta(S_{x_0^4x_1})\zeta(S_{x_0^6x_1}),\zeta(S_{x_0x_1})^3\zeta(S_{x_0^2x_1})^2, \zeta(S_{x_0x_1})^2\zeta(S_{x_0^2x_1})\zeta(S_{x_0^4x_1}),\\ 
 && \zeta(S_{x_0x_1})\zeta(S_{x_0^2x_1})\zeta(S_{x_0^6x_1}), \zeta(S_{x_0^2x_1})\zeta(S_{x_0^8x_1}), \zeta(S_{x_0x_1})\zeta(S_{x_0^4x_1})^2,\\
 &&   \zeta(S_{x_0x_1})\zeta(S_{x_0x_1^3x_0x_1^5}),\zeta(S_{x_0^3x_1x_0x_1^7})\}.
\end{eqnarray*}
\end{itemize}

We can see that these dimensions satisfy the following recurrence \cite{zagier}
\begin{eqnarray*}
d_1=0,d_2=d_3=1&\mbox{and}&\forall n\ge4,d_n=d_{n-2}+d_{n-3}.
\end{eqnarray*}  
This means that, up to weight $12$, our results obtained by the previous algorithms verify the Zagier's dimension conjecture.
As a consequence, this conjecture holds up to weight $12$ if and only if the irreducible polyzetas,
contained in each two following different lists, are algebraically independent (see \cite{VJM} for a discussion).
$$\begin{array}{|l||l||l|}
\hline
n& \text{irreducible polyzetas on }\{\Sigma_l\}_{l\in \Lyn Y} &\text{irreducible polyzetas on  }\{S_l\}_{l\in \Lyn X}\\
 \hline\hline
2& \zeta(\Sigma_{y_2})  &\zeta(S_{x_0x_1})\\
\hline
3&\zeta(\Sigma_{y_3})&\zeta(S_{x_0^2x_1})\\
\hline
4& & \\
\hline
5&\zeta(\Sigma_{y_5})&\zeta(S_{x_0^4x_1})\\
\hline
6& & \\
\hline
7&\zeta(\Sigma_{y_7})& \zeta(S_{x_0^6x_1})\\
\hline
8&\zeta(\Sigma_{y_3y_1^5})&\zeta(S_{x_0x_1^2x_0x_1^4})\\
\hline
9& \zeta(\Sigma_{y_9})&\zeta(S_{x_0^8x_1})\\
\hline
10& \zeta(\Sigma_{y_3y_1^7})&\zeta(S_{x_0x_1^2x_0x_1^6})\\
\hline
11& \zeta(\Sigma_{y_{11}}),\ \zeta(\Sigma_{y_2y_1^9})&
\zeta(S_{x_0^{10}x_1}),\ \zeta(S_{x_0x_1^2x_0x_1^2x_0x_1^4})\\
\hline
12&\zeta(\Sigma_{y_2^
 2y_1^8}),\ \zeta(\Sigma_{y_3y_1^9})&\zeta(S_{x_0x_1x_0x_1^9}),\ \zeta(S_{x_0^3x_1x_0x_1^7})\\
 \hline
\end{array}$$
\begin{center}
Table $2$: List of irreducible polyzetas up to weight $12$.
\end{center}

By Example \ref{different}, in general, one has
\begin{eqnarray*}
\forall l\in\Lyn Y,\pi_X(\Sigma_l)\neq S_{\pi_Xl}&\mbox{and}&
\forall l\in\Lyn X\setminus\{x_0\},\pi_Y(S_l)\neq\Sigma_{\pi_Yl}.
\end{eqnarray*}
This does not occur, due to a lemma by D. Perrin, with the Lyndon words themselves on which
$\{\zeta(l)\}_{l\in\Lyn Y}$ (or $\{\zeta(l)\}_{l\in\Lyn X}$) was provided in \cite{bigotte,MP,elwardi}.
Hence, we insist on the fact that $\{\zeta(\Sigma_l)\}_{l\in \Lyn Y}$ and $\{\zeta(S_l)\}_{l\in \Lyn X}$
provide two different systems of local coordinates and two lists of irreducible polyzetas (see Table $2$).

\section{Conclusion}
In the classical theory of (finite-dimensional) Lie groups, every ordered basis of the Lie algebra provides a system of local coordinates
of a suitable neighbourhood of the unity (of the group) via an ordered product of one-parameter groups corresponding to the (ordered) basis.

Here, we get a perfect analogue of this geometrical picture for the Hausdorff groups (in shuffle and stuffle Hopf algebras)
through Sch\"utzenberger's factorization. This does not depend on the regularization of shuffle and quasi-shuffle.

Moreover, through the bridge equation \eqref{regularization} which relates two elements on these groups and an identification
of the local coordinates of the L.H.S. and R.H.S. of \eqref{regularization2} which involve only convergent polyzetas as local
coordinates, we get, up to weight $12$,
\begin{itemize}
\item a confirmation of the Zagier's dimension conjecture,
\item two families of irreducible polyzetas ({\it i.e} two algebraic bases for polyzetas),
\end{itemize}
which are not due to the regularized double-shuffle relations (and we do not need any regularization).

This implementation will be used, in our forthcoming work, to determine the asymptotic expansions of harmonic sums
via Euler-Maclaurin formula.  
\nocite{DT, flajoletsalvy, CMM, IKZ, kaneko, orsay, BDMKT}
\section*{Bibliography}
\bibliography{mybib}{}
\bibliographystyle{plain}

\end{document}

%% file: macrosB.tex
\usepackage{amsmath}
\usepackage{amsfonts}
\usepackage{amssymb,amsthm}
\newcommand{\calL}{{\mathcal L}}

\newcommand{\calH}{{\mathcal H}}

\newcommand{\N}{{\mathbb N}}

\newcommand{\Z}{{\mathbb Z}}
\newcommand{\Q}{{\mathbb Q}}
\newcommand{\R}{{\mathbb R}}

\newcommand{\Frac}[2]{\displaystyle \frac{#1}{#2}}
\newcommand{\Sum}[2]{\displaystyle{\sum_{#1}^{#2}}}
\newcommand{\Prod}[2]{\displaystyle{\prod_{#1}^{#2}}}

\def\pol#1{\langle #1 \rangle}
\def\Lyn{{\mathcal Lyn}}

\gdef\stuffle{\;%
  \setlength{\unitlength}{0.0125cm}%
  \begin{picture}(20,10)(220,580)
  \thinlines
  \put(220,592){\line( 0,-1){ 10}}
  \put(220,582){\line( 1, 0){ 20}}
  \put(240,582){\line( 0, 1){ 10}}
  \put(230,592){\line( 0,-1){ 10}}
  \put(225,587){\line( 1, 0){ 10}}
  \end{picture}\;
}

\newtheorem{corollary}{Corollary}
\newtheorem{proposition}{Proposition}

\newtheorem{lemma}{Lemma}

\newtheorem{example}{Example}

\newcommand{\poly}[2]{#1 \langle #2 \rangle}

\def\QX{\poly{\Q}{X}}

\def\QY{\poly{\Q}{Y}}

\def\pol#1{\langle #1 \rangle}

\def\Sum{\displaystyle\sum}
\def\Prod{\displaystyle\prod}
\def\Frac{\displaystyle\frac}
\def\path{\rightsquigarrow}

\def\bv{\mid}

\gdef\minishuffle{{\scriptstyle \shuffle}}
\gdef\ministuffle{{\scriptstyle \stuffle}}